\documentclass[12pt]{amsart}
\pdfoutput=1
\usepackage[foot]{amsaddr}

\usepackage[utf8]{inputenc}
\usepackage{indentfirst}
\usepackage{amsmath}
\usepackage{amsthm}
\usepackage{amssymb}
\usepackage{amsfonts}
\usepackage{microtype}
\usepackage{fancyhdr}
\usepackage{tikz-cd}
\usetikzlibrary{arrows}
\usetikzlibrary{positioning}
\usepackage{xcolor}
\usepackage[colorlinks]{hyperref}
\usepackage{mathtools}
\usepackage{thmtools}

\usepackage{import}
\usepackage{xifthen}
\usepackage{pdfpages}
\usepackage{transparent}

\usepackage[
  margin=1in,
  includefoot,
  footskip=30pt,
]{geometry}

\usepackage{csquotes}

\usepackage[style=alphabetic, citestyle=alphabetic]{biblatex}
\hypersetup{
  colorlinks = true,
  citecolor = {purple!70!black},
  linkcolor = {red!60!black}
}

\newcommand{\kk}{\Bbbk}
\newcommand{\OO}{\mathcal{O}}
\newcommand{\PP}{\mathbb{P}}

\newcommand{\ZZ}{\mathbb{Z}}

\newcommand{\Spec}{\mathrm{Spec}\,}
\newcommand{\Ext}{\mathrm{Ext}}

\newcommand{\Hom}{\mathrm{Hom}}

\newcommand{\RHom}{\mathrm{RHom}}

\newcommand{\RGamma}{\mathrm{R}\Gamma\,}
\newcommand{\Pic}{\mathrm{Pic}}

\newcommand{\Coh}{\mathrm{Coh}}
\newcommand{\Dbcoh}{D^b_{\!\mathrm{coh}}}

\newcommand{\cA}{\mathcal{A}}

\newcommand{\emptyperp}{{}^\perp}

\newcommand{\iso}{\simeq}
\newcommand{\caniso}{\cong}

\newcommand{\monoarrow}{\hookrightarrow}
\newcommand{\epiarrow}{\twoheadrightarrow}

\newcommand{\im}{\operatorname{im}\,}

\newcommand{\supp}{\operatorname{supp}}

\newcommand{\Alb}{\operatorname{Alb}}

\newcommand{\red}{\operatorname{red}}

\declaretheoremstyle[
headformat=\NUMBER.\,\NAME\NOTE,
postheadspace=.5em,
spaceabove=6pt,
headfont=\normalfont\small\scshape,
notefont=\normalfont\small\mdseries, notebraces={(}{)},
bodyfont=\normalfont\itshape
]{plainswap}
\declaretheoremstyle[
headformat=\NAME\NOTE,
postheadspace=.5em,
spaceabove=6pt,
headfont=\normalfont\small\scshape,
notefont=\normalfont\small\mdseries, notebraces={(}{)},
bodyfont=\normalfont\itshape
]{nonumplainswap}
\declaretheoremstyle[
headformat=\NUMBER.\,\NAME\NOTE,
postheadspace=.5em,
spaceabove=6pt,
headfont=\normalfont\small\scshape,
notefont=\normalfont\mdseries, notebraces={(}{)},
bodyfont=\normalfont
]{definitionswap}
\declaretheoremstyle[
headformat=\NAME\NOTE,
postheadspace=.5em,
spaceabove=6pt,
headfont=\normalfont\itshape,
notefont=\mdseries, notebraces={(}{)},
bodyfont=\normalfont
]{myremark}

\declaretheorem[style=plainswap, name=Theorem, sharenumber=subsection]{theorem}

\declaretheorem[style=plainswap, numberlike=theorem, name=Proposition]{proposition}

\declaretheorem[style=plainswap, numberlike=theorem, name=Lemma]{lemma}
\declaretheorem[style=plainswap, numberlike=theorem, name=Corollary]{corollary}

\theoremstyle{definition}
\declaretheorem[style=definitionswap, numberlike=theorem, name=Definition]{definition}

\declaretheorem[style=definitionswap, numberlike=theorem, name=Example]{numberedexample}
\theoremstyle{myremark}
\newtheorem*{remark}{Remark}

\theoremstyle{remark}

\numberwithin{equation}{theorem}

\usepackage{quiver}

\title{On the support of admissible subcategories}
\author{Dmitrii Pirozhkov$^1$}
\address{$^1$Steklov Mathematical Institute of Russian Academy of Sciences, Moscow, Russia}
\email{dpirozhkov@mi-ras.ru}

\begin{document}

\begin{abstract}
Let $X$ be a smooth proper variety over an algebraically closed field of characteristic zero, and let $\mathcal{A} \subset \Dbcoh(X)$ be an admissible subcategory. Let $Z \subset X$ be the union of set-theoretical supports of all objects in $\mathcal{A}$ and assume that $Z \neq X$. We show that for any morphism from $Z$ to an abelian variety each fiber has no isolated points; this implies, for example, that $Z$ cannot be isomorphic to an abelian variety. The key input is the fact that while not all line bundles on $Z$ lift to infinitesimal thickenings of $Z$, sufficiently many do: specifically, we show that for any infinitesimal thickening $Z \subset \widetilde{Z}$ the restriction morphism $\mathrm{Pic}^0(\widetilde{Z}) \to \mathrm{Pic}^0(Z)$ on the connected components of Picard schemes induces an isogeny between Albanese group schemes of those connected components.
\end{abstract}
\maketitle

\section{Introduction}
\label{sec:introduction}

There are many examples of birational transformations of smooth algebraic varieties (blow-ups, some flops, some flips) for which we can explicitly track how the derived category of coherent sheaves changes~(\cite[Sec.~3]{bondal-orlov-survey}, \cite{kawamata-survey-1}, \cite{kawamata-survey-2}). These transformations produce many examples of semi-orthogonal decompositions of derived categories with a special property: the resulting decomposition has one ``main'' component, and all other components are admissible subcategories supported on subvarieties of lower dimension, related to the centers of the birational transformations. This natural family of interesting semi-orthogonal decompositions is a foundation for many deep theorems and alluring conjectures about derived categories of coherent sheaves.

The subvarieties contracted by birational morphisms between smooth projective varieties cannot be arbitrary; for instance, they are uniruled. The same is not necessarily true for the support of an arbitrary admissible subcategory in the derived category of coherent sheaves on a smooth projective variety. The standard example of the divergence between the minimal model program and the study of semi-orthogonal decompositions, namely the Enriques surface, can be used to construct an admissible subcategory whose support is not uniruled (see Example~\ref{ex:admissible_supported_at_enriques}).

However, the support of an admissible subcategory is not completely arbitrary. For example, any one-dimensional irreducible component of the support must be a rational curve~\cite{pirozhkov-curve-supported}. As a follow-up to that result, we prove the following constraint on the support of an admissible subcategory:

\begin{theorem}[{see Theorem~\ref{thm:main_theorem}}]
  \label{thm:main_theorem_intro}
  Let~$X$ be a smooth proper variety over an algebraically closed field~$\kk$ of characteristic zero and let~$\cA \subset \Dbcoh(X)$ be an admissible subcategory. Denote the (set-theoretical) support of~$\cA$ by~$Z \subset X$ and assume that~$Z \neq X$. Then for any morphism~$g\colon Z \to G$ from~$Z$ to an abelian variety~$G$ each fiber is a union of positive-dimensional connected components. Moreover, if~$Z^\prime \subset Z$ is an irreducible component whose intersection with other irreducible components of~$Z$ is zero-dimensional, then there are no generically finite morphisms from~$Z^\prime$ to an abelian variety.
\end{theorem}

The non-existence proved in Theorem~\ref{thm:main_theorem_intro} relies on the generalization of the $\Pic^0$-invariance of admissible subcategories proved originally by Kawatani and Okawa~\cite{kawatani-okawa}, adapted in a somewhat subtle way to Picard schemes of the infinitesimal thickenings of $Z$. This result is a generalization of the arguments used in \cite[Thm.~6.5]{pirozhkov-curve-supported} to show that any one-dimensional irreducible component of $Z$ is a rational curve. We heavily rely on the machinery developed in the previous paper~\cite{pirozhkov-curve-supported}.

The novel insight that allows us to work with higher-dimensional varieties is the following observation about line bundles on infinitesimal thickenings. In the one-dimensional case any line bundle on a reduced curve extends to a line bundle on an arbitrary infinitesimal thickening since the target of the obstruction map is a trivial abelian group. As soon as the dimension of $Z$ is at least two, obstructions in general exist. However, the obstruction map from $\Pic(Z)$ is not just a homomorphism of abelian groups: both the source and the target are actually (sets of $\kk$-points of) group schemes, and the map is a morphism of group schemes. Each connected commutative group scheme $G$ of finite type over a field has the largest proper quotient which is an abelian variety that we denote by $\Alb(G)$. We show that many line bundles in the connected component of the Picard scheme lift to arbitrary infinitesimal thickenings:

\begin{proposition}[{see Corollary~\ref{cor:picard_albanese_isogeny}}]
  \label{prop:picard_albanese_isogeny_intro}
  Let $\kk$ be an algebraically closed field of characteristic zero. Let $j\colon Z \monoarrow \widetilde{Z}$ be an infinitesimal thickening of proper schemes. Consider the restriction morphism $j^*\colon \Pic^0(\widetilde{Z}) \to \Pic^0(Z)$ on the connected components of Picard schemes. The induced map of Albanese group schemes
  \[ \Alb(\Pic^0(\widetilde{Z})) \to \Alb(\Pic^0(Z)) \]
  is an isogeny, i.e., a surjective map with finite kernel.
\end{proposition}

In particular, if $\Pic^0(Z)$ is an abelian variety, e.g., if $Z$ happens to be smooth, then the map $j^*\colon \Pic^0(\widetilde{Z}) \to \Pic^0(Z)$ is surjective, and all line bundles in the connected component of the Picard scheme lift to arbitrary infinitesimal thickenings.

\textbf{Outline}. In Section~\ref{sec:picard} we study line bundles on infinitesimal thickenings of proper schemes and prove Proposition~\ref{prop:picard_albanese_isogeny_intro}. In Section~\ref{sec:main_proof} we prove the main result of this note, Theorem~\ref{thm:main_theorem_intro}.

\textbf{Acknowledgments}. I thank Michel Brion for useful discussion regarding Prop.~\ref{prop:picard_albanese_isogeny_intro}. This work was performed at the Steklov International Mathematical Center and supported by the Ministry of Science and Higher Education of the Russian Federation (agreement no. 075-15-2025-303).

\section{Line bundles on infinitesimal thickenings}
\label{sec:picard}

\begin{definition}
  We say that a morphism $j\colon Z \monoarrow \widetilde{Z}$ of schemes is an \textit{infinitesimal thickening} if both schemes are of finite type over a field $\kk$, the morphism $j$ is a closed immersion, and the induced map on the reduced subschemes $j^{\red}\colon Z^{\red} \to \widetilde{Z}^{\red}$ is an isomorphism.
\end{definition}

As usual, we say that an infinitesimal thickening $j\colon Z \monoarrow \widetilde{Z}$ is a \textit{square-zero extension} if the ideal sheaf $\mathcal{I}_Z$ on $\widetilde{Z}$ corresponding to the closed subscheme $Z$ satisfies $\mathcal{I}^2 = 0$ as a subsheaf of $\OO_{\widetilde{Z}}$. Recall that the topological spaces underlying $Z$ and $\widetilde{Z}$ are canonically identified, and the square-zero condition implies, in particular, that the coherent sheaf $\mathcal{I}$ on $\widetilde{Z}$ is naturally equipped with the structure of a $\OO_Z$-module making it a coherent sheaf on $Z$.

In this section we study how line bundles on $Z$ differ from line bundles on $\widetilde{Z}$. We start with the basic observation.

\begin{lemma}[{see, e.g., \cite[Tag~0C6Q]{stacks-project}}]
  \label{lem:picard_les_for_square_zero}
  Let $j\colon Z \monoarrow \widetilde{Z}$ be a square-zero extension with ideal sheaf $\mathcal{I}$. Then there exists a long exact sequence of abelian groups
    \begin{equation}
      \label{eq:picard_les_for_square_zero}
      \begin{aligned}
            0 & \to H^0(Z, \mathcal{I}) \to H^0(\widetilde{Z}, \OO_{\widetilde{Z}}^*) \to H^0(Z, \OO_Z^*) \to \\
        & \to H^1(Z, \mathcal{I}) \to \Pic(\widetilde{Z}) \to \Pic(Z) \to \\
        & \to H^2(Z, \mathcal{I}) \to \dots
      \end{aligned}
    \end{equation}
\end{lemma}

\begin{proof}
  The restriction along $j$ induces a canonical morphism $\OO_{\widetilde{Z}}^* \to j_* \OO_Z^*$ of sheaves of abelian groups. It is easy to check that this map is a surjection of sheaves for any infinitesimal thickening. Since $j$ is a square-zero extension, the kernel of $\OO_{\widetilde{Z}}^* \to j_* \OO_Z^*$ is locally generated by invertible functions of the form $1 + f$ for a (nilpotent) function $f \in \mathcal{I}$. Thus the kernel of this map of sheaves of abelian groups happens to be isomorphic to the coherent sheaf $\mathcal{I}$.
\end{proof}

We are interested in the Picard groups not just as abelian groups, but as algebraic groups. The Picard scheme exists for any proper scheme:

\begin{theorem}[{Murre~\cite{murre-representability}, Oort~\cite{oort-picard}}]
  \label{thm:picard_exists}
  Let $X$ be a proper scheme over a field $\kk$. Then $\Pic(X)$ has a structure of a commutative group scheme over $\kk$, which is locally of finite type and has countably many connected components.
\end{theorem}

The long exact sequence (\ref{eq:picard_les_for_square_zero}) shows that the difference between $\Pic(Z)$ and $\Pic(\widetilde{Z})$ is controlled by the cohomology of the coherent sheaf $\mathcal{I}$ on $Z$. The group scheme structure on the cohomology space of a coherent sheaf is, naturally, just the usual vector space structure. This observation leads to the following result. For an overview of general notions related to (commutative) group schemes, such as the notion of a unipotent group, see, e.g., \cite{brion-up-to-isogeny}.

\begin{proposition}
  \label{prop:picard_difference_is_unipotent}
  Let $\kk$ be an algebraically closed field of characteristic zero. Let $j\colon Z \monoarrow \widetilde{Z}$ be an infinitesimal thickening of proper schemes. Consider the restriction morphism
  \[
    j^*\colon \Pic^0(\widetilde{Z}) \to \Pic^0(Z)
  \]
  on the connected components of Picard schemes. Then the kernel of $j^*$ is a unipotent algebraic group, and the cokernel of $j^*$ is a finite extension of a unipotent algebraic group.
\end{proposition}

\begin{proof}
  This is a variant of the argument from the classical paper~\cite{oort-picard}. Let $f\colon \widetilde{Z} \to \Spec \kk$ be the structure morphism. First we consider a special case: we assume that $j$ is a square-zero extension and that the base field $\kk$ is uncountable. Since any base change of a square-zero extension is a square-zero extension, the long exact sequence~(\ref{eq:picard_les_for_square_zero}) induces an exact sequence of sheaves on the large étale site of $\kk$:
  \[ R^1 f_* \mathcal{I} \to R^1 f_* \OO_{\widetilde{Z}}^* \to R^1 f_* \OO_{Z}^* \to R^2 f_* \mathcal{I} \]
  The cohomology of a coherent sheaf over a field commutes with base change, so the two étale sheaves $R^1 f_* \mathcal{I}$ and $R^2 f_* \mathcal{I}$ are represented by vector spaces $H^1(Z, \mathcal{I})$ and $H^2(Z, \mathcal{I})$. By Theorem~\ref{thm:picard_exists} two étale sheaves in the middle are also representable. By the Yoneda lemma this implies that we have an exact sequence of representing group schemes
    \begin{equation}
      \label{eq:picard_scheme_les}
      H^1(Z, \mathcal{I}) \to \Pic(\widetilde{Z}) \to \Pic(Z) \to H^2(Z, \mathcal{I})
    \end{equation}
  where all maps are homomorphisms of group schemes.

  The group scheme $H^1(Z, \mathcal{I})$ is connected, so its image is contained in $\Pic^0(\widetilde{Z})$. Thus the kernel of the restriction map $j^*\colon \Pic^0(\widetilde{Z}) \to \Pic^0(Z)$ is isomorphic to some quotient group scheme of $H^1(Z, \mathcal{I})$. By~\cite[Ex.~XVII, Prop.~2.2]{sga3-2} any algebraic quotient of a unipotent algebraic group is unipotent. Thus $\ker j^*$ is unipotent.

  Consider now two subgroups:
  \[ K := \ker(\Pic^0(Z) \to H^2(Z, \mathcal{I})) \quad \text{and} \quad K^\prime := \im(\Pic^0(\widetilde{Z}) \to \Pic^0(Z)). \]
  Both of them are (Zariski) closed subgroups of~$\Pic^0(Z)$: for~$K$ this is automatic since the kernel is the preimage of the closed point~$0 \in H^2(Z, \mathcal{I})$, for~$K^\prime$ we use the fact that the image of a \textit{quasi-compact} homomorphism of group schemes locally of finite type is closed~\cite[Ex.~VI${}_B$,~Prop.~1.2]{sga3-1}, and after restricting to the connected component~$\Pic^0(\widetilde{Z}) \subset \Pic(\widetilde{Z})$ the quasi-compactness condition is satisfied. The exactness of the sequence~(\ref{eq:picard_scheme_les}) implies that~$K^\prime$ is contained in~$K$ and that
  \[ K = \im (\Pic(\widetilde{Z}) \to \Pic(Z)) \cap \Pic^0(Z). \]

  By Theorem~\ref{thm:picard_exists} we know that~$\Pic(\widetilde{Z})$ has at most countably many connected components, so on the level of~$\kk$-points the index of the subgroup~$K^\prime (\kk) \subset K(\kk)$ is at most countable. However, since passing to the group of points over an algebraically closed field~$\kk$ of characteristic zero is an exact functor, this means that the finite type group scheme~$K / K^\prime$ has at most countably many~$\kk$-points. We have assumed that the base field~$\kk$ is uncountable, so this can happen only when the group scheme~$K / K^\prime$ is finite over~$\kk$. Then the cokernel of the morphism~$j^*\colon \Pic^0(\widetilde{Z}) \to \Pic^0(Z)$, namely the quotient~$\Pic^0(Z) / K^\prime$, is an extension of a closed algebraic subgroup of~$H^2(Z, \mathcal{I})$, which is unipotent by~\cite[Ex.~XVII, Prop.~2.2]{sga3-2}, by a finite group scheme~$K / K^\prime$, as we wanted to show.

  Now consider an arbitrary infinitesimal thickening $j\colon Z \monoarrow \widetilde{Z}$, still keeping the assumption that the base field $\kk$ is uncountable. Any such $j$ is a composition of square-zero extensions. An iterated extension of unipotent groups is itself a unipotent group by~\cite[Ex.~XVII, Prop.~2.2]{sga3-2}, so the claim about the kernel of $j^*$ is immediate. To deal with the cokernel, we need to show that an iterated extension of finite and unipotent commutative group schemes is itself a finite extension of a unipotent group. This follows from general results on the structure of commutative group schemes over fields of characteristic zero, such as \cite[Thm.~2.9(ii)]{brion-up-to-isogeny}; the key fact is that a connected commutative group admitting a finite map into a unipotent group is itself unipotent.

  Reducing an arbitrary algebraically closed field to an uncountable one is easy: the properties of being connected, being finite over the base field, and being a unipotent group scheme are all stable under the extension of algebraically closed fields, see, e.g.,~\cite[Ex.~XVII, Prop.~1.2]{sga3-2}.
\end{proof}

By the structure theory of commutative group schemes, each connected commutative group scheme has the largest abelian variety quotient~\cite[Lem.~4.5]{brion-up-to-isogeny}, referred to as the Albanese group scheme. Note that some sources use the term Albanese group scheme in a similar context to refer to the largest \textit{semi-abelian} quotient, i.e., an extension of an abelian variety by a group of multiplicative type; we don't consider groups of multiplicative type here.

\begin{corollary}
  \label{cor:picard_albanese_isogeny}
  Let $\kk$ be an algebraically closed field of characteristic zero. Let $j\colon Z \monoarrow \widetilde{Z}$ be an infinitesimal thickening of proper schemes. Consider the restriction morphism
  \[
    j^*\colon \Pic^0(\widetilde{Z}) \to \Pic^0(Z)
  \]
  on the connected components of Picard schemes. The induced map of Albanese group schemes
  \[ \Alb(\Pic^0(\widetilde{Z})) \to \Alb(\Pic^0(Z)) \]
  is an isogeny, i.e., a surjective map with finite kernel.
\end{corollary}

\begin{proof}
  By \cite[Prop.~4.8]{brion-up-to-isogeny} the Albanese functor on connected commutative group schemes is exact up to isogeny, meaning that a complex of connected commutative group schemes whose cohomology group schemes are finite transforms to a complex of abelian group schemes whose cohomology group schemes are finite after passing to Albanese group schemes. Since the Albanese scheme of a unipotent group scheme is zero, Proposition~\ref{prop:picard_difference_is_unipotent} implies that the map
  \[ \Alb(\Pic^0(\widetilde{Z})) \to \Alb(\Pic^0(Z)) \]
  has finite kernel and finite cokernel. However, since the group scheme $\Alb(\Pic^0(Z))$ is connected, it does not admit a nontrivial map into a finite group scheme. Thus the cokernel is zero.
\end{proof}

In fact Brion proved a similar exactness up to isogeny statement for the group schemes of multiplicative type as well~\cite[Prop.~4.4]{brion-up-to-isogeny}, which implies that the restriction morphism~$\Pic^0(\widetilde{Z}) \to \Pic^0(Z)$ is surjective modulo unipotent groups. We omit this stronger argument here. For the sake of completeness, let us show an explicit example where the cokernel is indeed a nontrivial unipotent group. There are simple examples where~$Z$ itself is a non-reduced scheme, but we present below an example in which~$Z$ is reduced and irreducible.

In order to do this, recall some basic notions from deformation theory. We say that a square-zero extension $j\colon Z \monoarrow \widetilde{Z}$ is \textit{locally trivial} if Zariski-locally there exists an isomorphism of sheaves of algebras $\OO_{\widetilde{Z}} \iso \OO_{Z} \oplus \mathcal{I}$ (with the natural square-zero multiplication structure on the right-hand side). Since algebra automorphisms of the direct sum $\OO_Z \oplus \mathcal{I}$ that commute with the quotient map $\OO_Z \oplus \mathcal{I} \epiarrow \OO_Z$ are easily seen to be derivations of $\OO_Z$ with values in the coherent sheaf $\mathcal{I}$, by a Cech cohomology argument we obtain that locally trivial square-zero extensions with ideal sheaf $\mathcal{I}$ are classified by the space $\Ext^1(\Omega^1_Z, \mathcal{I})$.

\begin{numberedexample}
  Let~$X$ be a proper variety with a class~$h \in H^{1}(X, \OO_X)$ such that the morphism
  \[ H^{1}(X, \OO_X) \xrightarrow{d} H^{1}(X, \Omega^1_X) \]
  induced by the Kahler differential map~$\OO_{X} \to \Omega^1_{X}$ sends~$h$ to a nonzero element~$d(h)$. Let~$Y$ be a proper variety with~$H^1(Y, \OO_Y) \neq 0$. Consider the product variety~$Z := X \times Y$ with projection morphisms~$\pi_X\colon X \times Y \to X$ and~$\pi_Y\colon X \times Y \to Y$. Let~$\xi \in \Ext^1_{Z}(\Omega^1_Z, \Omega^1_Z)$ be the image of the class~$\pi_X^*(h) \otimes \mathrm{id}_{\Omega^1_Z}$ under the multiplication map
  \[ H^1(Z, \OO_Z) \otimes \Hom_{Z}(\Omega^1_Z, \Omega^1_Z) \to \Ext^1_{Z}(\Omega^1_Z, \Omega^1_Z). \]
  Let~$j\colon Z \monoarrow \widetilde{Z}$ be the locally trivial square-zero thickening of~$Z$ with the ideal sheaf~$\mathcal{I} := \Omega^1_Z$ defined via the class~$\xi$. Then the map~$\Pic^0(\widetilde{Z}) \xrightarrow{j^*} \Pic^0(Z)$ has a positive-dimensional cokernel.
\end{numberedexample}

\begin{remark}
  We can take $X$ to be a cuspidal cubic curve $C = \{ y^2 z = x^3 \} \subset \PP^2$: using the standard affine cover
  \[ \begin{aligned}
    U & = \Spec \kk[t^{-1}]   &                                           \\
      &                     & \quad \supset U \cap V = \Spec \kk[t, t^{-1}] \\
    V & = \Spec \kk[t^2, t^3]
  \end{aligned} \]
  one can check that the class $h \in H^1(C, \OO_C)$ represented by the Cech cocycle
  \[ t \in \Gamma(U \cap V, \OO_{U \cap V}), \]
  is sent to
  the element $\mathrm{d} t \in \Gamma(U \cap V, \Omega^1_{U \cap V})$ which is not a coboundary and thus represents a nonzero class $d(h) \in H^1(C, \Omega^1_C)$. We can take $Y$ to be a smooth genus one curve to get a two-dimensional example $Z \caniso C \times Y$.
\end{remark}

\begin{proof}
  First, let us consider a general setting. Let~$\xi \in \Ext^1(\Omega^1_Z, \mathcal{I})$ be some class and let~$j\colon Z \monoarrow \widetilde{Z}$ be the corresponding locally trivial square-zero thickening. The short exact sequence
  \[ 0 \to \mathcal{I} \to \OO_{\widetilde{Z}} \to \OO_Z \to 0 \]
  of coherent sheaves on~$\widetilde{Z}$ induces an exact sequence
  \[ H^{1}(\widetilde{Z}, \OO_{\widetilde{Z}}) \to H^1(Z, \OO_Z) \to H^2(Z, \mathcal{I}). \]
  Since the first cohomology space of the structure sheaf is canonically identified with the tangent space to the Picard scheme~\cite[Thm.~5.11]{kleiman-picard}, the map on Picard schemes has a positive-dimensional cokernel if and only if the map
  \[ H^1(Z, \OO_Z) \to H^2(Z, \mathcal{I}) \]
  of vector spaces is nonzero. Using the local Cech cover argument one can check that this map is nothing but the multiplication by the class~$d^*(\xi) \in H^1(Z, \mathcal{I})$ obtained as the pullback of~$\xi \in \Ext^1(\Omega^1_Z, \mathcal{I})$ along the Kahler differential morphism~$d\colon \OO_Z \to \Omega^1_Z$.

  Since~$Z \caniso X \times Y$ is a direct product, there exists an isomorphism
  \[ \Omega^1_Z \caniso \pi_X^* \Omega^1_X \oplus \pi_Y^* \Omega^1_Y. \]
  By the projection formula we have an isomorphism
  \[
    H^*(Z, \Omega^1_Z) \caniso (H^*(X, \Omega^1_X) \otimes H^*(Y, \OO_Y)) \oplus (H^*(X, \OO_X) \otimes H^*(Y, \Omega^1_Y)).
  \]
  We want to show that the multiplication by~$d(\xi)$ induces a nonzero morphism
  \[ H^1(Z, \OO_Z) \to H^2(Z, \Omega^1_Z). \]
  From the definition of~$\xi$ in terms of the class~$h \in H^1(X, \OO_X)$ it is easy to check that the composition
  \[
    H^0(X, \OO_X) \otimes H^1(Y, \OO_Y) \monoarrow H^1(Z, \OO_Z) \xrightarrow{d(\xi)} H^2(Z, \Omega^1_Z) \epiarrow H^1(X, \Omega^1_X) \otimes H^1(Y, \OO_Y)
  \]
  is equal to the map
  \[ H^0(X, \OO_X) \otimes H^1(Y, \OO_Y) \xrightarrow{d(h) \otimes \mathrm{id}} H^1(X, \Omega^1_X) \otimes H^1(Y, \OO_Y), \]
  which is nonzero by the assumptions that~$d(h) \neq 0$ and~$H^1(Y, \OO_Y) \neq 0$.
\end{proof}

\section{Proof of the main theorem}
\label{sec:main_proof}

In this section we prove the main result of this note. Recall that if~$X$ is a smooth proper variety, for an admissible subcategory~$\cA \subset \Dbcoh(X)$ the \textit{support} of~$\cA$ is defined as~$\supp(\cA) := \cup_{A \in \cA} \supp(A)$; it is a Zariski closed subset since it equals the support of any classical generator of~$\cA$. See \cite{pirozhkov-curve-supported} for a more detailed discussion.

\begin{theorem}
  \label{thm:main_theorem}
  Let $X$ be a smooth proper variety over an algebraically closed field $\kk$ of characteristic zero and let $\cA \subset \Dbcoh(X)$ be an admissible subcategory. Denote the (set-theoretical) support of $\cA$ by $Z \subset X$ and assume that $Z \neq X$. Then for any morphism $g\colon Z \to G$ from $Z$ to an abelian variety $G$ each fiber is a union of positive-dimensional connected components. Moreover, if $Z^\prime \subset Z$ is an irreducible component whose intersection with other irreducible components of $Z$ is zero-dimensional, then there are no generically finite morphisms from $Z^\prime$ to an abelian variety.
\end{theorem}

In particular, the support of an admissible subcategory can never be isomorphic to an abelian variety. The second claim of the theorem does not follow from the first since a generically finite morphism~$Z^\prime \to G$ to an abelian variety cannot, in principle, be extended to a well-defined morphism~$Z \to G$ even if the intersection of~$Z^\prime$ with other irreducible components is a finite set of points.

\begin{proof}
  As noted in the introduction, this is a direct generalization of \cite[Thm.~6.5]{pirozhkov-curve-supported} using Corollary~\ref{cor:picard_albanese_isogeny} instead of \cite[Lem.~4.5]{pirozhkov-curve-supported}. The proof proceeds via a close examination of the properties of the projections of skyscraper sheaves to the subcategory~$\cA \subset \Dbcoh(X)$. Let~$p \in Z \subset X$ be a point and let $\OO_p \in \Dbcoh(X)$ be the skyscraper sheaf at $p$. Consider the projection triangle for $\OO_p$:
  \[ B_p \to \OO_p \to A_p \to B_p [1], \]
  where~$A_p \in \cA$ and~$B_p \in \emptyperp \cA$. As explained in~\cite[Lem.~3.8]{pirozhkov-curve-supported}, the object~$A_p$ is not zero since the point~$p$ lies in the support of~$\cA$; in fact, the support of~$A_p$ contains the point~$p$. By~\cite[Cor.~3.4]{pirozhkov-curve-supported} there exists a closed subscheme~$i\colon \widetilde{Z} \monoarrow X$ whose reduced part is contained in~$Z$ and an object~$E \in \Dbcoh(\widetilde{Z})$ such that~$A_p \iso i_*(E)$. By \cite[Prop.~6.3]{pirozhkov-curve-supported} we know that all cohomology sheaves of~$E$ are invariant under the action of the group scheme~$\Pic^0(\widetilde{Z})$. We denote by~$j\colon Z \monoarrow \widetilde{Z}$ the embedding of the reduced subscheme.

  We start by considering the first claim of the theorem. Assume that there exists a morphism~$g\colon Z \to G$ to an abelian variety~$G$ and a point~$p \in Z$ which is an isolated point of the fiber~$g^{-1}(g(p)) \subset Z$. Then we claim that for any object~$E \in \Dbcoh(\widetilde{Z})$ whose support contains~$p$ and whose cohomology sheaves are~$\Pic^0(\widetilde{Z})$-invariant the point~$p$ is an isolated point of~$\supp(E)$. We argue by contradiction. If~$p \in \supp(E)$ is not an isolated point, then for some~$m \in \ZZ$ the support of the cohomology sheaf~$\mathcal{H} := \mathcal{H}^{m}(E) \in \Coh(\widetilde{Z})$ contains a positive-dimensional irreducible component passing through~$p$. Then there exists a reduced and irreducible proper curve~$C \subset Z$ such that~$p \in C$ and (set-theoretically)~$C \subset \supp(\mathcal{H})$. Let~$\widetilde{C} \to C$ be the resolution of singularities of the curve, and denote by~$c\colon \widetilde{C} \to Z$ the composed map to~$Z$. Note that the morphism~$c$ is finite. Let~$\mathcal{F} \in \Coh(\widetilde{C})$ be the non-derived pullback of~$\mathcal{H}$ along the composition of finite morphisms
  \[ \widetilde{C} \xrightarrow{c} Z \monoarrow^{j} \widetilde{Z}. \]
  By construction~$\supp(\mathcal{F}) = \widetilde{C}$. Since, as explained above, the sheaf~$\mathcal{H} \in \Coh(\widetilde{Z})$ is~$\Pic^0(\widetilde{Z})$-invariant, its pullback~$\mathcal{F} \in \Coh(\widetilde{C})$ is invariant under twists by the line bundles in the image of the composition
    \begin{equation}
      \label{eq:picard_restriction_to_curve}
      \Pic^0(\widetilde{C}) \xleftarrow{c^*} \Pic^0(Z) \xleftarrow{j^*} \Pic^0(\widetilde{Z}).
    \end{equation}

  Let us show that this is a nontrivial constraint on the sheaf~$\mathcal{F}$, i.e., that the restriction map~$\Pic^0(\widetilde{Z}) \to \Pic^0(\widetilde{C})$ is not zero. First, observe that since~$\widetilde{C}$ is a smooth proper curve, the group scheme~$\Pic^0(\widetilde{C})$ is an abelian variety. Thus, by the universal property of the Albanese group scheme, the morphism~(\ref{eq:picard_restriction_to_curve}) factors as
  \[ \Pic^0(\widetilde{C}) \leftarrow \Alb(\Pic^0(Z)) \leftarrow \Alb(\Pic^0(\widetilde{Z})) \leftarrow \Pic^0(\widetilde{Z}). \]
  The rightmost map is surjective by definition, and Corollary~\ref{cor:picard_albanese_isogeny} shows that the middle map is surjective as well. Thus the composition~(\ref{eq:picard_restriction_to_curve}) is nonzero if and only if the map~$\Alb(\Pic^0(Z)) \to \Pic^0(\widetilde{C})$ is nonzero, or, equivalently, if~$\Pic^0(Z) \xrightarrow{c^*} \Pic^0(\widetilde{C})$ is nonzero. To prove this, we use the assumption that~$p$ is an isolated point of the fiber of~$g\colon Z \to G$. Consider a sequence of morphisms
  \[ \widetilde{C} \xrightarrow{c} Z \xrightarrow{g} G. \]
  Since the map~$c\colon \widetilde{C} \to Z$ is not constant and its image contains the point~$p \in Z$ which is an isolated point in the fiber of~$g$, the map~$\widetilde{C} \to G$ is finite. A finite map from a smooth projective curve~$\widetilde{C}$ to an abelian variety~$G$ induces a nonzero map~$\Pic^0(G) \to \Pic^0(\widetilde{C})$ by the universal property of Albanese variety of~$\widetilde{C}$. Since this nonzero map factors through the map~$c^*\colon \Pic^0(Z) \to \Pic^0(\widetilde{C})$, the morphism~$c^*$ is also nonzero. Hence the composition~(\ref{eq:picard_restriction_to_curve}) is also not zero. The image of a nonzero homomorphism of connected group schemes is necessarily a positive-dimensional subgroup scheme of~$\Pic^0(\widetilde{C})$.

  To reach the desired contradiction with the assumption that $p \in \supp(E)$ is not an isolated point, it remains to show that a coherent sheaf $\mathcal{F} \in \Coh(\widetilde{C})$ with full support $\supp(\mathcal{F}) = \widetilde{C}$ cannot be invariant under any positive-dimensional subset of $\Pic^0(\widetilde{C})$. The argument from~\cite[Lem.~6.6]{pirozhkov-curve-supported} applies: the determinant of the torsion-free quotient of $\mathcal{F}$ would be a line bundle on $\widetilde{C}$ invariant under a positive-dimensional subset of $\Pic^0(\widetilde{C})$, which is impossible. Thus we conclude that the support of $\mathcal{F}$ must be a finite union of points, which contradicts the assumption that $C \subset \supp(E)$ is an irreducible curve passing through the point $p$.

  Now we are done: if~$p \in \supp(E)$ is an isolated point, then the support of the object~$A_p \iso i_*(E) \in \cA \subset \Dbcoh(X)$ also includes~$p$ as an isolated point. However, since we have assumed that~$Z = \supp(\cA)$ is strictly smaller than~$X$, for any object in the admissible subcategory~$\cA \subset \Dbcoh(X)$ the set-theoretical support is a union of positive-dimensional connected components by~\cite[Lem.~3.10]{pirozhkov-curve-supported}. This is a contradiction. We conclude that the fibers of any morphism~$g\colon Z \to G$ to an abelian variety have no isolated points.

  The second claim of the theorem is a simple variation of the argument above. Assume that~$Z^\prime \subset Z$ is an irreducible component which intersects all other irreducible components of~$Z$ in a zero-dimensional subset. Let~$f\colon Z^\prime \monoarrow Z$ be the closed embedding. Pick a point~$p \in Z^\prime$ that does not lie on other irreducible components and such that the map~$g\colon Z^\prime \to G$ is finite at~$p$. Then to adapt the argument above we only need to show that the restriction morphism
  \[ f^*\colon \Pic^0(Z) \to \Pic^0(Z^\prime) \]
  induces a surjection on the Albanese group schemes. However, as explained in \cite[Lem.~4.6]{pirozhkov-curve-supported}, the map~$f^*$ itself is surjective: any line bundle~$L$ on~$Z^\prime$ can be glued with a trivial line bundle on the union of other irreducible components of~$Z$, and this gluing trick can be performed in families, which implies the surjectivity on the connected components of the Picard scheme. Since passing to the Albanese group scheme is an exact-up-to-isogeny functor, the map~$\Alb(f^*)\colon \Alb(\Pic^0(Z)) \to \Alb(\Pic^0(Z^\prime))$ is also surjective. The conclusion we would get is that the support of the object~$E$ contains~$p$ as an isolated point, which is also ruled out by~\cite[Lem.~3.10]{pirozhkov-curve-supported}.
\end{proof}

\begin{remark}
  The condition that $Z$ admits a finite morphism into an abelian variety does not impose any structural constraints on the group scheme $\Pic^0(Z)$. It is easy to construct examples where $\Pic^0(Z)$ has nontrivial unipotent and multiplicative components.
\end{remark}

To indicate the limits on what can be proved about the support of an admissible subcategory, we present an example where an admissible subcategory is supported on a non-uniruled smooth subvariety.

\begin{numberedexample}
  \label{ex:admissible_supported_at_enriques}
  Let~$S$ be an Enriques surface and let~$\mathcal{L}$ be an acyclic line bundle on~$S$, i.e., a line bundle satisfying~$H^*(S, \mathcal{L}) = 0$. Such acyclic line bundles exist on any Enriques surface~(\cite[Thm.~2]{zube-enriques} or \cite[Prop.~2.1]{borisov-nuer}); if~$S$ contains a smooth rational~$(-2)$-curve~$C$, then the line bundle~$\mathcal{L} := \OO_{S}(-C)$ works. Consider the embedding~$i\colon S \monoarrow X := \PP_{S}(\mathcal{L} \oplus \OO_S)$ into the completed total space of the line bundle~$\mathcal{L}$. This is a smooth hypersurface in a smooth threefold with normal bundle isomorphic to~$\mathcal{L}$. Then using the acyclicity of~$\mathcal{L}$ one can easily check that
  \[ \RHom_{X}(i_* \OO_S, i_* \OO_S) \caniso \RHom_{S}(i^* i_* \OO_S, \OO_S) \caniso \RGamma(S, \OO_S) \caniso \kk [0]. \]
  Thus~$i_* \OO_S \in \Dbcoh(X)$ is an exceptional object. It generates an admissible subcategory of~$\Dbcoh(X)$ whose support is the Enriques surface~$S \subset X$.
\end{numberedexample}

\printbibliography
\end{document}